\documentclass[letterpaper, 10 pt, conference]{ieeeconf}  %
\newif\ifDraft
\Draftfalse

\input{preamble}
\Detailstrue
\Arxivtrue

\begin{document}

\maketitle
\thispagestyle{empty}
\pagestyle{empty}

\begin{abstract}
	We propose a simple safety filter design for stochastic discrete-time systems
based on piecewise affine probabilistic control barrier functions,
providing an appealing balance between modeling flexibility and
computational complexity.
Exact evaluation of the safety filter consists of solving a
mixed-integer quadratic program (MIQP) if the dynamics are control-affine,
(or a mixed-integer nonlinear program in general).
We propose a heuristic search method that replaces this by a
small number of small-scale quadratic programs (QPs), or nonlinear programs (NLPs)
respectively.
The proposed approach provides a flexible framework in which arbitrary (data-driven)
quantile estimators can be used to bound the probability of safety violations.
Through extensive numerical experiments,
we demonstrate improvements in conservatism and computation time
with respect to existing methods,
and we illustrate the flexibility of the method for modeling complex
safety sets.
Supplementary material can be found at \href{https://mathijssch.github.io/ecc26-supplementary/}{mathijssch.github.io/ecc26-supplementary/}.

\end{abstract}

\section{Introduction}
In recent years, there has been a tremendous increase in
data-driven and learning-based controller design.
Leveraging recent advances in machine learning, as well as
the ubiquity of data and high-fidelity simulators,
these approaches often achieve remarkable levels of performance,
in particular for complex tasks.
However, oftentimes, these controllers only come with very
limited safety guarantees \cite{brunke_SafeLearningRobotics_2022}.

This has motivated an increased research interest in so-called \textit{safety filters},
which essentially operate by projecting the control action
of some given policy onto a set of admissible inputs
designed to guarantee specific safety properties of the system.
The key idea is to preserve as much of the baseline controller’s performance as possible,
enforcing safety through only minimal and necessary interventions.

At the core of designing such a safety filter lies the \ac{CBF}.
\acp{CBF}, originally coined in \cite{wieland_ConstructiveSafetyUsing_2007} for continuous-time systems,
are functions whose 0-superlevel set is a controlled invariant set.
They play a role analogous to control Lyapunov functions
(used extensively in the theory of \ac{MPC}),
but with a focus on ensuring safety rather than stability \cite{ames_ControlBarrierFunctions_2019}.

Over the years, several variants of \acp{CBF} have emerged, each extending the
original framework to handle different settings \cite{wang2025safesurv,li2023survey}.
Discrete-time \acp{CBF} were introduced in \cite{agrawal_discrete_2017} as a
counterpart to continuous-time \acp{CBF}.
For uncertain systems, robust \acp{CBF} \cite{breeden2023robust,liu2025robust} were introduced to handle bounded disturbances. These approaches are suitable
when the support of the disturbance is compact and can be tightly estimated.
Whenever this is not the case, such robust formulations tend to be either inapplicable
(e.g., for unbounded noise) or overly conservative (no tight estimates of the support).
This motivates stochastic and
\acp{PCBF}, which aim to bound the probability of leaving the safe set
over a finite horizon (the $N$\textit{-step exit probability}).
Although conceptually, these approaches allow us to deal with unbounded noise,
they require very careful design that ensures that the stage-wise exit
probability is tightly bounded, in particular for long horizons $N$,
where conservative bounds may easily lead to vacuous estimates of the $N$-step exit probability.

A common approach to the design of \ac{PCBF} is to enforce a supermartingale
condition at every time step and leverage concentration inequalities, such as Ville's
\cite{cosner_RobustSafetyStochastic_2023,cosner_generative_2023} or Freedman's
inequalities \cite{cosner_bounding_2024}. These methods require moment
information and often result in conservative bounds. More recent work attempts to
reduce this conservatism by introducing auxiliary functions that reshape the
expectation condition, though these approaches remain restricted to Gaussian
noise models \cite{fushimi_SafetyCriticalControlDiscretetime_2025}. Risk-based
formulations using CVaR \cite{wang_safe_2025,kishida_risk-aware_2024} offer an
alternative but likewise are prone to conservatism.

More recent work has relaxed distributional assumptions entirely,
using data-driven estimators such as conformal prediction or the scenario approach \cite{mestres_ProbabilisticControlBarrier_2025}.
While appealing for its generality, it is very challenging for such an approach
to provide tight guarantees without exploiting any additional problem structure.

Our goal is therefore to develop a probabilistic safety filter that is able to exploit
as much as possible known structure of the \ac{PCBF} (ensuring tightness of the violation probabilities),
while maintaining sufficient flexibility in the \ac{PCBF} class to model complex
safe set geometries.
In general, non-smooth \acp{CBF} (including piecewise-constant \acp{CBF} \cite{mazouz_piecewise_2025})
have been identified as useful
to balance flexibility and structure \cite{cavorsi2020tractable}.
In this work, we will restrict our attention in particular to piecewise affine
\ac{CBF} candidates, which can model arbitrary safe sets, potentially even allowing
the use of representations based on neural networks, with ReLU activation functions.

We highlight the following contributions of this work:
\begin{inlinelist*}
	\item We synthesize a safety filter using a probabilistic PWA \ac{CBF} condition requiring knowledge only
	of the disturbance quantiles
	\item We introduce a heuristic solution method to speed up evaluation of the safety filter
	\item We extend the formulation to the data-driven setting, where the distribution of
	the disturbance is not known, but rather estimated from data
	\item We present extensive numerical experiments comparing our method with existing approaches,
    demonstrating its effectiveness, reduced conservativeness and flexibility.
\end{inlinelist*}

\section*{Notation}
For $k \in \N$, let
\( [k]  = \{ 1, 2, \dots, k \}\).
Let $\xi: \Omega \to \Xi$ denote a random variable on a probability
space $(\Omega, \F, \prob)$.
As is customary, we will often denote
an event $\{ \omega \in \Omega \mid \xi(\omega) \in A \} \in \F$
by the shorthand $\{ \xi \in A \}$.
We denote by
$q_{\delta}(Z) = \inf \{ \tau \mid \prob\{ Z \leq \tau \} \geq \delta \}$
the $\delta$-quantile of random variable $Z: \Omega \to \R$, 
and by $\supp(Z) = \itrsect \{ \cl(C) \mid \prob\{Z \in C\}=1\}$
the support of $Z$.
We denote by $\borel(\R^n)$ the Borel sigma-algebra on $\R^n$.
The notation $T: \R^n \rightrightarrows \R^m$ indicates that
$T$ is a set-valued mapping, i.e., $T(x) \subseteq \R^m$, $x \in \R^n$.
The (effective) domain of $T$ is defined as
$\dom T \dfn \{ x \in \R^n \mid T(x) \neq \emptyset \}$.
We denote by $\compl{C} = \R^n \setminus C$,
the complement of a set $C \in \R^n$
and by $\lev_{\geq \gamma} h = \{ x \in \R^n \mid h(x) \geq \gamma \}$
the $\gamma$-superlevel set of function $h: \R^n \to \R$.
$v_{(k)}$, $k \in [n]$ denotes the $k$-th order statistic of $(v_i)_{i\in[n]}$,
i.e., $v_{(1)} \leq v_{(2)} \leq  \dots \leq v_{(n)}$.

\section{Problem Statement} \label{sec:problem-statement}
Consider a nonlinear, discrete-time dynamical system governed by dynamics
\begin{equation}\label{eq:dynamics}
	x_{k+1} = F(x_k, u_k, \xi_k) = f(x_k, u_k) + g(x_k, \xi_k), \; k \in \N,
\end{equation}
where $x_k \in \R^\ns$, $u_k \in U \subseteq \R^\na$ are the state and input at time $k$, respectively,
and $(\xi_k)_{k \in \N}$ is a stochastic process on a filtered probability
space $(\Omega, \F, (\F_k)_{k \in \N}, \prob)$,
taking values on $\R^{\nxi}$.
We assume $U$ is a nonempty closed, convex set.
We will assume $(u_k)_{k \in \N}$ is adapted to the filtration $(\F_k)_k$,
that is, it depends \textit{causally} on the randomness in the system.
Furthermore, we will assume that $\xi_{k}$ is $\F_{k+1}$-measurable
(which implies its value is not revealed until time step $k + 1$).
\ifExtended
\begin{remark}
	A sufficient condition for $\F_{k+1}$-measurability of $\xi_k$ is to have state measurements,
	and $\xi \mapsto g(x, \xi)$ invertible for all $x \in \R^\ns$.
	However, in practice, there may be other ways of measuring $\xi_k$,
	so there is no need to make this stronger assumption.
\end{remark}
\fi
Finally, for simplicity,
we assume that $(\xi_k)_{k \in \N}$ is an i.i.d.\ sequence,
(i.e., $\xi_k$ is independent of $\F_k$ for all $k \in \N$).
This allows us to define $\prob_{\xi}$, the common law of $\xi_k$:
\[
  \prob\{\xi_k \in E\} = \prob_{\xi}(E), \quad
	E \in \borel(\R^{\nxi}), \; k \in \N.
\]

Let $S \subseteq \R^{\ns}$ be a set of admissible
states, commonly referred to as a \textit{safe set}.
Ideally, we would like to synthesize a control policy
$\pi = (\kappa_k)_{k \in \N}$ with $\kappa_k: \R^{\ns} \to \R^{\na}$,
which ensures that the state $x_k$ governed by \eqref{eq:dynamics}
with $u_k = \kappa_k(x_k)$
remains within the set $S$ almost surely.
Although there are special cases in which this can
be guaranteed (for instance, when $g(x_k, \xi_k)$ is restricted to a compact set),
this is not an achievable goal for general unbounded noise.

Therefore, we instead consider a probabilistic variant of this objective,
similar to \cite{fushimi_SafetyCriticalControlDiscretetime_2025,mestres_ProbabilisticControlBarrier_2025}.
We will consider finite control horizons and consider the following condition for failure.
\begin{definition}[$N$-step exit probability]
	The $N$-step exit probability for system \eqref{eq:dynamics}
	and initial state $x \in S$ under policy $\pi$ is given by
	\[
		P_{N}(x; \pi) \dfn \prob \big\{ \exists k \in [N] : x_{k} \notin S\big\}.
	\]
	When $\pi = (\kappa_k)_{k \in \N}$ is time-invariant, (i.e., $\kappa_k = \kappa, \forall k \in \N$)
	we will slightly abuse
	notation and simply write $P_{N}(x; \kappa)$.
\end{definition}
In summary, our objective is to ensure that the $N$-step exit probability
remains bounded:
\begin{problem} \label{prob:main-problem}
Given an initial state $x \in S$, and an allowable exit probability
$\epsilon \in (0, 1)$, synthesize a policy $\pi = (\kappa_k)_{k \in \N}$ under which
\begin{equation}\label{eq:exit-prob-condition}
	P_{N}(x; \pi) \leq \epsilon.
\end{equation}
\end{problem}

We will address \cref{prob:main-problem} by means of
probabilistic control barrier functions.
\Acp{CBF} have gained significant popularity in recent years, in large part due to
their flexibility and computational simplicity, especially for deterministic systems.
Recently, their usage has been extended to stochastic systems with (potentially unbounded)
noise.

\begin{definition}[Probabilistic \acs{CBF} {\cite[Def. 2]{mestres_ProbabilisticControlBarrier_2025}}] \label{def:prob-cbf}
	Fix $\delta \in (0, 1)$.
	Function $h: \R^{\ns} \to \R$ is a $\delta$-probabilistic \ac{CBF}
	($\delta$-\acs{PCBF})
	for system \eqref{eq:dynamics} and safe set $S = \lev_{\geq 0}h$,
	if there exists $\alpha \in [0, 1]$, such that
	\begin{equation}\label{eq:p-cbf}
		\forall x \in S, \exists u \in U: \prob_{\xi}\{ h(F(x,u,\xi)) \geq \alpha h(x) \} \geq 1 - \delta.
	\end{equation}
\end{definition}

\begin{remark}
	Technically, \cref{def:prob-cbf}
	could be generalized to allow non-i.i.d. disturbances and time-varying
	tolerances $\delta$. However, since we have already restricted our scope to i.i.d. disturbances,
	this provides little added value in our setting.
\end{remark}

\begin{proposition}[{\cite[Prop. 1]{mestres_ProbabilisticControlBarrier_2025}}] \label{prop:relation-delta-epsilon}
  Fix $N \in \N$, and let $h$ be a $\delta$-\acs{PCBF} with
	\begin{equation}\label{eq:condition-eps-delta-const}
		\delta \leq 1 - (1 - \epsilon)^{\frac{1}{N}}.
	\end{equation}
	Then, for any policy $\pi = (\kappa_k)_{k\in [N]}$, satisfying \eqref{eq:p-cbf}
	with $u_k = \kappa_k(x_k)$, for all $k \in [N]$, 
  $P_{N}(x, \pi) \leq \epsilon$ for all $x \in S$.
\end{proposition}

Given a probabilistic \ac{CBF} $h$, \cref{prob:main-problem} can thus be
restated as finding a policy $\pi$ that satisfies the conditions of \cref{prop:relation-delta-epsilon}

The remainder of this paper is organized as follows.
\Cref{sec:verification}
will be dedicated to verifying \eqref{eq:p-cbf}
online for general piecewise affine \acp{CBF} in the case where
quantiles of linear functions of $g(x, \xi)$
can be computed exactly (given $x$).
In \cref{sec:data-driven},
we will relax this assumption, extending our results to
data-driven quantile estimates.
In \cref{sec:numerical}, we provide
extensive numerical experiments, comparing our approach to existing methods.

\section{Verification of piecewise affine CBFs} \label{sec:verification}
We now direct our attention to the numerical verification of
the probabilistic \ac{CBF} condition \eqref{eq:p-cbf}
for piecewise affine functions $h$. Without loss of generality \cite{ovchinnikov_MaxMinRepresentationPiecewise_2002},
we will assume that $h$ is represented as,
\begin{equation}\label{eq:pw-lin-barrier}
	h(x) = \min_{i \in [\np]} h_{i}(x) \text{ with } h_i(x) = \max_{j \in [\nf_i]} c_{ij}^\top x - b_{ij}.
\end{equation}
Observe that the corresponding safe set $S = \lev_{\geq 0} h$ can be constructed as
the complement of $\np$ convex polyhedral sets, where the
$i$-th polyhedron has $\nf_i$ facets.
Indeed, consider a collection $\{ P_i \}_{i=1}^{\np}$ of (open) polyhedral sets
\[
	P_i = \{ x \in \R^{\ns} \mid C_i x < b_i \}, \text{ with } C_i \in \R^{\nf_i} \times \R^{\ns}, b_i \in \R^{\nf_i},
\]
where $C_i$ is a matrix with rows $c_{ij} \in \R^{\ns}$, $j \in [\nf_i]$.
Then $x \notin P_{i} \iff h_i(x) \geq 0$, $i \in [\np]$
and therefore,
\(
x \notin \union_{i \in [\np]} P_i \iff h(x) \geq 0.
\)
Thus,
\(
S = \compl{\big(\union_{i \in [\np]} P_i\big)}.
\)

\subsection{Sufficient conditions} \label{sec:sufficient-conditions}

In order to obtain a sufficient condition for \eqref{eq:p-cbf},
let us first restrict our attention to a single set $P_i$.
For ease of notation, we will define the following short-hands
\begin{subequations}\label{eq:shorthands}
	\begin{align}
		\Dh(x,u,\xi)     & \dfn \min_{i \in [\np]} \Dh_i(x,u,\xi),                      \\
		\Dh_i(x, u, \xi) & \dfn h_i(F(x,u,\xi)) - \alpha h(x), \label{eq:d-h}          \\
		\bt_{ij}(x)      & \dfn b_{ij} - \alpha h(x), \quad x \in \R^n, \label{eq:d-b}
	\end{align}
\end{subequations}
so $h_i(F(x,u, \xi)) \geq \alpha h(x)$ if and only if
\(
\Dh_i(x, u,\xi) \geq 0
\).
We do not explicitly reflect the dependence of $\Dh_i$ and $\bt_{ij}$ on $\alpha$
in the notation, since $\alpha$ will be a constant, chosen \textit{a priori}.

Now we have the following elementary results (proofs
are provided in \cref{sec:appendix} for completeness.)
\begin{lemma} \label{lem:bound-single}
  For all $x \in \R^\ns$, $u \in \R^\na$, $i \in [\np]$,
	\[
    \prob_{\xi}\{ \Dh_i(x, u, \xi) < 0 \} \leq \min_{j \in [\nf_i]} \prob_{\xi}\{ c_{ij}^\top F(x, u, \xi) < \bt_{ij}(x) \}.
	\]
\end{lemma}

From \cref{lem:bound-single}, we easily obtain
a bound on the overall violation probability of any of the constituents $P_i$.

\begin{lemma} \label{lem:union-bound}
	For any $x \in \R^\ns$, $u \in \R^{\na}$,
	\begin{equation*}\label{eq:union-bound}
		\prob_{\xi} \{ \Dh(x, u, \xi) < 0 \}
		\leq
		\sum_{i \in [\np]} \prob_{\xi} \{ \Dh_i(x, u, \xi) < 0 \}.
	\end{equation*}
	If the sets $P_i$, $i \in [\np]$ are pairwise disjoint, then
	this holds with equality.
\end{lemma}

Combining \cref{lem:bound-single,lem:union-bound},
we obtain a sufficient condition for the \ac{CBF} condition \eqref{eq:p-cbf},
which boils down to a tightening of the linear constraints by means of a
control-independent quantile. 
\begin{proposition} \label{prop:combined-bound}
	Let $x \in \R^\ns$, $u \in U$, $\delta \in (0, 1)$,
	and $\delta^{(i)} \in (0, 1)$, $i \in [\np]$,
	with $\sum_{i \in [\np]} \delta^{(i)} \leq \delta$.
	If for each $i \in [\np]$, there exists a $j \in [\nf_i]$,
	such that
	\begin{equation} \label{eq:constraint}
		c_{ij}^\top f(x,u) - \bt_{ij}(x) \geq 
    -\fullquant{x, \xi}
	\end{equation}
	then,
	\(
	\prob_\xi\{ h(F(x,u,\xi)) < \alpha h(x) \} \leq \delta.
	\)
\end{proposition}

\begin{proof}
	Using \eqref{eq:shorthands}, we have for any $i \in [\np]$,
	\begin{align*}
		\nonumber
		                               & \prob_{\xi}\{ h_i(F(x,u,\xi)) < \alpha h(x) \} \leq \delta^{(i)}                                      \\
		\nonumber
		\labelrel{\impliedby}{impl:1}  & \exists j \in [\nf_i]: \prob_{\xi}\big\{c_{ij}^\top F(x,u,\xi) < \bt_{ij}(x) \big\} \leq \delta^{(i)} \\
		\nonumber
		\labelrel{\impliedby}{equiv:2} & \exists j \in [\nf_i]: \eqref{eq:constraint} \text{ holds.}
	\end{align*}
	Here, implication \eqref{impl:1} follows from \cref{lem:bound-single}.
	To show implication \eqref{equiv:2},
	let $E_{ij}^- = \{\xi \mid c_{ij}^\top F(x,u,\xi) < \bt_{ij}(x) \} \subseteq
		E_{ij} = \{\xi \mid c_{ij}^\top F(x,u,\xi) \leq \bt_{ij}(x) \}$, so
    \( \prob( E_{ij}^- ) \leq \prob ( E_{ij}) \).
  If  \eqref{eq:constraint} holds, then by definition of the quantile
  $\quantile$, $\prob\{ E_{ij} \} \leq \delta^{(i)}$.
	Therefore, if for all $i \in [\np]$, there exists $j \in [\nf_i]$
	such that \eqref{eq:constraint} holds, then by \cref{lem:union-bound}
	and \eqref{eq:shorthands},
	\[
		\prob_\xi \{ h(F(x, u, \xi)) < \alpha h(x) \} \leq \tsum_{i \in [\np]} \delta^{(i)} \leq \delta.\qedhere
	\]
\end{proof}

\begin{remark}
	Note that implication \eqref{equiv:2} in the paper becomes an equality for continuous distributions.
	For discrete and mixed distributions, \eqref{eq:constraint} can be slightly improved by
	replacing $-\quantile[\delta^{(i)}]{(c_{ij}^\top g(x, \xi))}$ with $\quantile[1 - \delta^{(i)}]{(-c_{ij}^\top g(x,\xi))}$.
\end{remark}

It follows that by enforcing condition
\eqref{eq:constraint} %
at every time step $k \in [N]$, we can control the $N$-step exit probability.
For many standard distributions, the quantiles involved in \eqref{eq:constraint}
can be easily computed. 
For more general approaches, we propose an extension that replaces the quantile by
data-driven estimators in \cref{sec:data-driven}.

\begin{corollary} \label{cor:specs}
	Fix $\epsilon, \delta \in (0,1)$ such that \eqref{eq:condition-eps-delta-const} holds.
	If a policy $\pi = (\kappa_k)_{k \in [N]}$ exists, such that
	for all $x \in S$ and $k \in [N]$,
	the conditions of \cref{prop:combined-bound}
	hold with $u=\kappa_k(x)$, %
	then $h$ is a $\delta$-probabilistic \ac{CBF} and $P_N(x; \pi) \leq \epsilon$.
\end{corollary}
\begin{proof}
	Let $(x_k)_{k \in [N]}$ be the random state sequence generated by \eqref{eq:dynamics},
	with $u_k = \kappa_k(x_k)$.
	If $x_k \in S$, then, by assumption, \eqref{eq:constraint} holds.
	\Cref{prop:combined-bound} then implies
	that
	\[
		\begin{aligned}
			     & \prob\{ h(F(x_k, \kappa_k(x_k), \xi_k)) < \alpha h(x_k) \mid x_k \in S \}  \\
			= \; & \prob_\xi \{ h(F(x_k, \kappa_k(x_k), \xi)) < \alpha h(x_k) \} \leq \delta.
		\end{aligned}
	\]
	The claim then follows from applying \cref{prop:relation-delta-epsilon}.
\end{proof}

\subsection{Safety filter design}

Informed by the results in \cref{sec:sufficient-conditions},
we now proceed with the design of a policy satisfying
the conditions of \cref{cor:specs}.
Typically, a policy enforcing \acp{CBF}-based constraints
is implemented by means of a
so-called \textit{safety filter}: For a given state $x$, the control
action is computed by projecting the inputs of a given base policy at $x$
onto the set of inputs satisfying the \ac{CBF} conditions.
Although many different objectives for the control actions could be
considered without changing the safety properties of the controller, we will
adopt this construction here for simplicity.

More precisely, given the current state $x$, a base policy
$\bar{\kappa}: \R^n \to \R^m$,
and tolerances $\delta^{(i)}$, $i \in [\np]$,
\eqref{eq:constraint}
we define the \textit{safety filter policy} as
$\kappa(x) \in \Pi(x)$ with
\begin{align} \label{eq:safety-filter-general}
	\Pi(x) = \argmin_{u \in U} & \; \nrm{ u - \bar{\kappa}(x)}^2                                                        \\
	\nonumber\sttshort   & \; \max_{j \in [\nf_i]} c_{ij}^\top f(x,u) - q_{ij}(x) \geq 0, \; \forall i \in [\np], 
\end{align}
where
\begin{equation}\label{eq:quantile-shorthand}
	q_{ij}(x) =  \bt_{ij}(x) - \fullquant{x,\xi}.
\end{equation}
In this work, we will assume $h$ to be given and well-designed, in
the sense that $\dom \Pi \supseteq S$,
ensuring that $\kappa$ is well-defined for all $x \in S$.

Problem \eqref{eq:safety-filter-general}
can be formulated as a mixed integer program by the introduction
of binary variables $s_{ij} \in \{0, 1\}$:
\begin{subequations} \label{eq:mip}
	\begin{align}
		\minimize_{u \in U, s_{ij} \in \{0, 1\}} \; & \nrm{ u - \bar{\kappa}(x)}^2                                            \\
		\sttshort\;\;
		                                                   & s_{ij} ( q_{ij}(x) -c_{ij}^\top f(x,u) ) \leq 0,  \label{binary}        \\
		                                                   & \tsum_{j' \in [\nf_i]} s_{ij'} \geq 1, \quad i \in [\np], j \in [\nf_i]
	\end{align}
\end{subequations}
If $f$ is control-affine, i.e. $f(x, u) = A(x) + B(x) u$, 
and $U$ is polyhedral,
then problem \eqref{eq:mip}
can be reformulated as a \ac{MIQP}
via the well-known big-M reformulation:
Replacing \eqref{binary} with
$\tilde{q}_{ij}(x) - c_{ij}^\top B(x) u \leq M_{ij} (1 - s_{ij})$,
where $\tilde{q}_{ij} = q_{ij} - c_{ij}^\top A(x)$, and $M_{ij} \gg 0$ some sufficiently large real numbers.

In general, however, problem \eqref{eq:mip} is a
mixed-integer nonlinear program, which may challenge the
real-time capabilities of the method,
in particular on embedded hardware.

To address this potential computational bottleneck,
we now describe a simple alternative approach.
The goal is to provide the same safety guarantees at a lower
computational cost.
According to \cref{cor:specs}
it suffices to associate with each $i \in [\np]$,
an index $j \in [\nf_i]$, such that \eqref{eq:constraint} holds
for the applied control action $u$.

Let $J$ denote the set of all possible assignments of an index
$j \in [\nf_i]$ to each $i \in [\np]$, i.e.,
\begin{equation}\label{eq:J-set}
	J = \{ \jj: \N^{\np} \mid j_i \in [\nf_i], i \in [\np] \}.
\end{equation}
Given such a $\jj = (j_i)_{i\in[\np]} \in J$,
let
\begin{subequations} \label{eq:safety-filter-approx}
	\begin{align}
		\hat{\Pi}_{\jj}(x) = \argmin_{u \in U} & \; \nrm{ u - \bar{\kappa}(x)}^2                                   \\
		\sttshort                        & \; c_{ij_i}^\top f(x,u) \geq q_{ij_i}(x), \; \forall i \in [\np].
		\label{eq:constraints-sf-approx}
	\end{align}
\end{subequations}
where $q_{ij}(x)$ is given by \eqref{eq:quantile-shorthand}.
By construction, any $u \in \hat{\Pi}_j(x)$ satisfies \eqref{eq:constraint}.
Therefore, the problem reduces to finding a $\jj \in J$, such that
$\hat{\Pi}_{\jj}(x)$ is nonempty.

\begin{algorithm}
	\caption{Piecewise-affine probabilistic safety filter}
	\label{alg:main-method}
	\begin{algorithmic}[1]
		\Require $\epsilon, \delta \in (0,1)$ satisfying \eqref{eq:condition-eps-delta-const},
		base policy $\bar\pi=(\bar\kappa_k)_{k\in[N]}$
		\For{$k=1,\dots,N$}
		\State $J_k \gets J$
		\State Observe state $x_k$
		\State Select $(\delta^{(i)})_{i\in[\np]}$ s.t. $\sum_{i \in [\np]}\delta^{(i)} \leq \delta$ \label{line:select-deltas} \Comment{Cf. \S\ref{sec:risk}}
		\While{$J_k \neq \emptyset$} \label{line:while}
		\State Select $\jj \in J_k$ and $J_k \gets J_k \setminus \{\jj\}$ \label{line:select-j} \Comment{Cf. \S\ref{sec:index-search}}
		\State $U_k = \hat{\Pi}_{\jj}(x_k)$  \Comment{Cf. \eqref{eq:safety-filter-approx}}
		\If{$U_k \neq\emptyset$}
		\State \label{line:projection}$u_k \in U_k$
		\State \textbf{break}
		\ElsIf{$J_k = \emptyset$}
		\Return Infeasible at time $k$
		\EndIf
		\EndWhile
		\EndFor
	\end{algorithmic}
\end{algorithm}

This conceptual procedure is summarized in \cref{alg:main-method}.
Note that some freedom remains in the exact implementation
of lines \ref{line:select-deltas}
and \ref{line:select-j}.
Even though they do not affect the guarantee of \cref{cor:specs},
they can greatly affect the practical performance of the method (both in
terms of computational cost and conservatism).

\subsubsection{Risk allocation} \label{sec:risk}
\Cref{alg:main-method} does not
specify how the violation tolerances $\delta^{(i)}$, $i \in [\np]$ that enter
problem \eqref{eq:safety-filter-approx} through \eqref{eq:quantile-shorthand}, should be chosen.
\Cref{cor:specs} only requires their sum to be below $\delta$,
but permits us to freely choose them.
For simplicity, in the rest of this paper we take
\begin{equation}\label{risk_alloc}
\delta^{(i)} = \frac{\delta}{\np}, \; \forall i \in [\np].
\end{equation}
In future work, we will explore better ways to allocate this risk, as this simple formulation may introduce unnecessary conservatism.

\subsubsection{Search over index $j$} \label{sec:index-search}
Although \eqref{eq:safety-filter-approx} will very often
be a small-scale optimization problem (a convex \ac{QP} in the case of control-affine
dynamics),
the overall complexity of \cref{alg:main-method}
is also determined by the number of iterations in the inner loop
(cf. line \ref{line:while} in \cref{alg:main-method}).
Since $|J| = \prod_{i \in [\np]}\nf_i$, the worst-case complexity can grow
exponentially in the number of constituent polyhedra $P_i$
(one exception being when $\nf_i = 1$ for all $i$, i.e., $P_i$ are halfspaces).

Fortunately, the number of effective trials can be vastly
reduced in practice by checking the most promising candidates first.
We propose the following heuristic for iterating over the set $J$.
Under the assumption that $f(x,u)$ will tend to be close to $x$,
we can approximate constraints \eqref{eq:constraints-sf-approx} by
the corresponding sections of the piecewise affine barrier function
\begin{equation}\label{eq:approximate-constraint}
	h_{ij}(x) = c_{ij}^\top x - b_{ij} \geq 0.
\end{equation}
It then stands to reason that if
\eqref{eq:approximate-constraint}
is satisfied for all $(i, j_i)_{i\in[\np]}$
(for some $\jj = (j_i)_{i \in [\np]} \in J$),
then we also expect $\hat{\Pi}_{\jj}(x)$ to be nonempty.
Thus, the first candidate sequence %
should consist of the halfspaces of each polyhedron $P_i$,
for which constraint violation in the current state $x$ is lowest.
To this end,
let $\sigma_i= (\sigma_{ik})_{k \in [\nf_i]}$, $i \in [\np]$, denote the permutation of the index set
$[\nf_i]$ that sorts the sequence $(h_{ij}(x))_{j \in [\nf_i]}$
in decreasing order, i.e.,
\(
\hat{h}_{i\sigma_{ik}} \geq \hat{h}_{i\sigma_{ik+1}}\;  \forall k \in [\nf_i-1], i \in [\np].
\)

In case of infeasibility, we then first proceed by
testing the other halfspaces of the set $P_i$ that has the largest
violation, since we expect to gain most from swapping a constraint that $x$
already (nearly) violates.
That is, we sort the indices $i \in [\np]$ in
decreasing order of $h_i(x)$ (cf. \eqref{eq:pw-lin-barrier}),
resulting in the sequence $\ii = (i_k)_{i \in [\np]}$, so that
\(
h_{i_{k}}(x) \geq h_{i_{k+1}}(x)
\)
We then obtain successive candidate sequences $(j_1, \dots, j_{\np}) \in J$ by iterating
over the set $\bigtimes_{i \in \ii} \sigma_i$ in lexicographic order,
ensuring that we check that halfspaces of $P_{i_{\np}}$ (corresponding to the largest violations)
in the innermost loop.
This procedure is illustrated in \cref{fig:sorting}. In the depicted configuration,
the search procedure first checks all facets of $P_1$ before moving on to the
next candidate for $P_2$.

Despite this greedy search method being a heuristic, we have found empirically
that it allows us to find a feasible collection of indices within one or two
inner iterations. For this reason it is useful in practice to limit
the number of inner iterations to some small amount.
\begin{remark}
	We have observed comparable performance by replacing $x$ by $f(x, \bar{\kappa}(x))$
	(i.e., the predicted successor state under the base policy),
	in the construction above.
\end{remark}
\begin{figure}[htb]
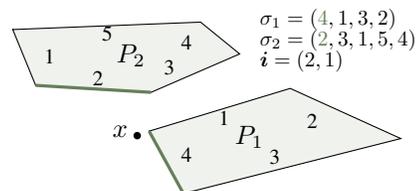

	\centering
	\includetikz{example-ordering}
	\caption{Example of the search heuristic described in \cref{sec:index-search}.
		At the first iteration, $\jj = (4, 2)$ would be selected, as shown in the figure,
		since this combination provides the least violation in the current state.
		In the case of infeasibility, the next candidates are
		$(1, 2), (3,2), (2,2), (4,3), (4, 3), (1, 3), \dots, (2, 4).$%
	}
	\label{fig:sorting}
\end{figure}

\section{Data-driven extensions} \label{sec:data-driven}
Since the disturbance distribution is often unknown in practice,
requiring the knowledge of its quantile $\quantile$
in \eqref{eq:quantile-shorthand} may be restrictive.
In this scenario, we can replace the quantile in the right-hand side of
condition \eqref{eq:constraint} by a data-driven estimator,
to obtain a \textit{probabilistic \ac{CBF} with high-confidence}.
This then allows us to provide the same guarantees as before, at the cost
of some additional conservatism to account for the lack of knowledge of the
distribution. For readability, we defer some longer proofs to \cref{sec:appendix}.

\subsection{General case}

Let \(\D=(\xi_t)_{t \in [\nsamp]}\) denote a dataset, drawn independently from
$\prob_\xi$. Note that $\mathcal{D}$ is distributed
according to the product measure $\prob_{\xi}^{\nsamp}$. Let 
$\Dsp \dfn \supp(\D)$ be the support of $\D$.

\begin{definition}[{\cite[Def. 3]{mestres_ProbabilisticControlBarrier_2025}}]
	Let $\delta, \gamma \in (0 ,1)$.
	Function $h$ is a \ac{PCBF} with high confidence, denoted $(\delta, \gamma)$-\acs{PCBF} if
	for all $x \in S$, there exists $\D$-measurable $u \in U$, such that
	\begin{equation} \label{eq:cbf-hc}
		\prob_{\xi}^\nsamp \Big\{
		\prob_{\xi}\big\{ h(F(x,u,\xi)) \geq \alpha h(x) \big\} \geq 1 - \delta
		\Big\} \geq 1 - \gamma.
	\end{equation}
\end{definition}

To establish sufficient conditions for \eqref{eq:cbf-hc},
we utilize the following basic result.
\begin{lemma}\label{lem:quantile-bound}
	Let $\delta, \gamma \in (0, 1)$, be given, and
	let $Z_\tau$, $\tau \in [\nsamp]$ be i.i.d. scalar random
	variables, with law $\prob_Z$.
	If
  \begin{equation}\label{eq:dmin}
		\nsamp > \minsmp(\gamma, \delta) \dfn \tfrac{\log(\gamma)}{\log(1 - \delta)},
	\end{equation}
	then $\tau = \Fb^{-1}(\gamma; \nsamp, \delta) \in [\nsamp]$, and
	\[
    \prob_{Z}^\nsamp\{ Z_{(\tau)} \leq \quantile[\delta](Z) \} \geq 1 - \gamma,
	\]
	where $\Fb^{-1}(\gamma; \nsamp, \delta)$
	denotes the $\gamma$-quantile of the binomial distribution $\bin(\nsamp, \delta)$.
\end{lemma}

Given $\D$, we construct for each $i \in [\np]$, $j \in [\nf_i]$, the sample
\begin{equation}\label{eq:def-Zij}
	Z^{ij}(x) = (Z_{t}^{ij}(x))_{t \in [\nsamp]} \text{ with }
	Z_t^{ij}(x)= c_{ij}^\top g(x,\xi_t).
\end{equation}
With this notation in place, \cref{lem:quantile-bound}
implies the following.

\begin{proposition} \label{prop:data-driven-ij}
	Let $\gamma,\delta \in (0, 1)$, $x \in \R^n$ be fixed,
	and let $\nftot \dfn \sum_{i \in [\np]} \nf_i$ denote the sum of the
	number of halfspaces of the polyhedra $P_i$, and allocate 
\(\delta^{(i)}\) as in \eqref{risk_alloc}. If
	\begin{equation}\label{eq:nsample-1}
		\nsamp > \max_{i \in [\nf]} \minsmp(\tfrac{\gamma}{\nftot}, \delta^{(i)}),
	\end{equation}
	then $\tau_{i} = \Fb^{-1}(\tfrac{\gamma}{\nftot}; \nsamp, \delta^{(i)}) \in [\nsamp]$, $i \in [\np]$,
	and
	\[
		\prob_{\xi}^{\nsamp}
		\big\{
		Z_{(\tau_i)}^{ij}(x) \leq
		\fullquant{x, \xi}, i \in [\np], j \in [\nf_i]
		\big\}
		\geq 1 - \gamma.
	\]
\end{proposition}

This provides a data-driven counterpart to \cref{cor:specs}:
\begin{lemma} \label{cor:conditions-dd}
  Let $\kappa_{\D}: \R^{\ns} \to U$ be a data-dependent control law.
  If for all $\D \in \Dsp$,
	and for all $x \in S$, there exists $\jj = (j_i)_{i \in [\nf]} \in J$ with
	\begin{equation} \label{eq:constraint-dd}
		c_{ij_i}^\top f(x,\kappa_\D(x)) - \bt_{ij_i}(x) \geq -Z^{ij_i}_{(\tau_i)}(x),
	\end{equation}
	where $\tau_i$ is defined as in \cref{prop:data-driven-ij};
	Then, $h$ given by \eqref{eq:pw-lin-barrier} is a
	$(\delta, \gamma)$-\acs{PCBF}.
\end{lemma}
\begin{proof}
	Fix $x \in S$,
	and define the event $E_x = \{ \D : Z_{(\tau_i)}^{ij} \leq \fullquant{x, \xi} \}$.
	If $\D \in E_x$, then \eqref{eq:constraint-dd}
	implies \eqref{eq:constraint}.
	Therefore, by \cref{prop:combined-bound},
	\(
	\prob_{\xi}^{\nsamp} \big\{
	\prob_\xi\{ h(F(x,\kappa_\D(x),\xi)) \geq \alpha h(x) \} \geq 1 - \delta.
	\big\} \geq \prob_{\xi}^{\nsamp} (E_x).
	\)
	Finally, by \cref{prop:data-driven-ij}, $\prob_{\xi}^{\nsamp} (E_x) \geq 1 - \gamma$,
	for all $x \in S$.
\end{proof}
From this, we finally obtain the desired guarantee.
\begin{proposition}[{\cite[Prop. 6]{mestres_ProbabilisticControlBarrier_2025}}] \label{prop:dd-guarantee-general}
	Fix $\epsilon \in (0,1)$.
	If $h$ is a $(\delta, \gamma)$-\acs{PCBF}, with $\delta$ satisfying
	\eqref{eq:condition-eps-delta-const}
	and $\gamma \leq \nicefrac{\tilde{\gamma}}{N}$, then, for any (data-dependent) control law
  $\kappa_{\D}: \R^n \to U$, satisfying \cref{cor:conditions-dd},
	\begin{equation} \label{eq:safety-hc}
		\prob_{\xi}^\nsamp \Big\{
		P_N(x, \kappa_\D) \leq \epsilon
		\Big\} \geq 1 - \tilde{\gamma}, \; \forall x \in S.
	\end{equation}
\end{proposition}

\subsection{Special case: \texorpdfstring{$g(x,\xi) = \xi$}{g(x,xi)=xi}}
In many practical applications, it can be assumed that the disturbance is
independent of the state, i.e., $g(x, \xi) = \xi$. In this case, we can obtain 
significantly improved results, in particular for long horizons $N$.
Note that $Z^{ij}$ defined in \eqref{eq:def-Zij} is now
independent of $x$.
We have the following alternative to
\cref{prop:dd-guarantee-general}.

\begin{proposition} \label{prop:data-driven-ij-indep}
	Fix $\epsilon, \tilde{\gamma},\delta \in (0,1)$, with $\delta$ satisfying
	\eqref{eq:condition-eps-delta-const}.
  Let $\kappa_\D: \R^{\ns} \to U$ be a data-dependent control law.
  If for all $\D \in \Dsp$, and
	$x \in S$, there exists $\jj = (j_i)_{i \in [\nf]} \in J$ with
	\begin{equation} \label{eq:constraint-dd-indep}
		c_{ij_i}^\top f(x,\kappa_\D(x)) - \bt_{ij_i}(x) \geq -Z^{ij_i}_{(\tau_i)},
	\end{equation}
	where $\tau_i$ is defined as in \cref{prop:data-driven-ij},
  \(\delta^{(i)}\) satisfying \eqref{risk_alloc} and \(\gamma=\tilde{\gamma}\),
	then \eqref{eq:safety-hc} holds.

\end{proposition}
\begin{proof}
	Define the event $E = \{ \D:
		Z_{(\tau_i)}^{ij} \leq
		\quantxi{\xi}, i \in [\np], j \in [\nf_i]
		\}$.
	If $\D \in E$, then \eqref{eq:constraint-dd-indep} implies
	\eqref{eq:constraint} for all $x \in S$.
	Thus, by \cref{cor:specs},
	$\D \in E \implies P_{N}(x_0; \kappa_\D) \leq \epsilon,\, \forall x \in S$.
	Therefore, 
  \[
    \prob_{\xi}^\nsamp (P_{N}(x_0;\kappa_\D) \leq \epsilon) \geq \prob_{\xi}^\nsamp(E) \geq 1 - \tilde{\gamma},
  \]
	where the last inequality follows from \cref{prop:data-driven-ij}.
\end{proof}
Note that the main difference between 
\cref{prop:dd-guarantee-general,prop:data-driven-ij-indep}
is that the former requires a tightening of the confidence level: $\gamma \leq \tfrac{\tilde{\gamma}}{N}$,
whereas the latter does not.
The reason for this is that the now state-\textit{independent} quantile estimates
only have to be computed once, whereas in general,
it is computed again for each $x_k$, $k \in [N]$. Therefore,
the estimator could fail to lower-bound the true quantile at every time step $k$,
necessitating the condition on $\gamma$ to compensate.

\section{Numerical Examples} \label{sec:numerical}
We now present several numerical experiments demonstrating the
features of the proposed method and comparing it to existing methods
in the literature.
All optimization problems are solved using the \cvxpy{} \cite{diamond2016cvxpy} interface unless stated otherwise.
In all experiments, we set $\alpha = 0$ in our method.
\subsection{Quadruped navigation through a narrow corridor}  \label{sec:corridor}
We consider a setup similar to
\cite{cosner_RobustSafetyStochastic_2023,mestres_ProbabilisticControlBarrier_2025},
where the goal is to navigate a quadrupedal robot through a narrow corridor.
Let $R_{z}(\theta) \in \R^{3\times 3}$ represent the rotation matrix corresponding 
to a rotation of angle $\theta$ around the $z$ axis in $\R^3$.
The robot is modeled by the dynamics
\begin{equation}
	x_{k+1} = x_{k} + \Delta t\, R_z(\theta_k) u_k + \xi_k ,
\end{equation}
where \(x_k =  (p^x_k, p^y_k, \theta_k) \) consists
of the planar position [\si{\meter}] and the heading angle,
and $u_k = (v_k^x, v_k^y, \omega_k)$ contains setpoints for
velocity (in $x$ and $y$) [\si{\meter\per\second}] and the
angular velocity [\si{\radian\per\second}]
of the robot (around the $z$ axis), respectively.
The safe set is given by $S = \{ (p^x, p^y, \theta) \mid |p^y| \leq 0.5 \}$,
which in \cite{cosner_RobustSafetyStochastic_2023,mestres_ProbabilisticControlBarrier_2025} 
is modeled using the quadratic \ac{CBF} \(\hQ(x) = 0.5^{2} - (p^y)^{2}\).
Since our approach is specialized to piecewise-linear functions, 
we instead use \(\hpwl(x) = \min\{-p^y + 0.5,\; p^y + 0.5 \}\) as our 
candidate \ac{CBF},
(using $\np = 2$ polyhedra $P_i$).
As the base policy we use \(\bar{\kappa}: (p^x, p^y, \theta) \mapsto (0.2,\,1,\,-\theta) \),
which was deliberately chosen to be unsafe (i.e., driving the nominal system towards the walls of
the corridor) to emphasize the effects of the safety filter.
Following \cite{mestres_ProbabilisticControlBarrier_2025},
we use
\(\xi_k\sim\mathcal{N}(0,\Sigma)\), with
\(\Sigma=\diag(0,\sigma^2,0)\), \(\Delta t=\SI{0.1}{\second}\) and \(N=20\).

\subsubsection{Known disturbance distribution}

Assuming the distribution of $\xi$ is known, our method can be compared
to the martingale approach outlined in
\cite{cosner_RobustSafetyStochastic_2023} and the method proposed by Fushimi et al.
\cite[Thm. 11]{fushimi_SafetyCriticalControlDiscretetime_2025}, which we
refer to as the \textit{\cosner{} method} and the \textit{\fushimi{} method},
respectively. Both methods require the selection of some hyperparameters, as 
we now describe.

The \cosner{} approach refers specifically to \cite[eq.~(ED)]{cosner_RobustSafetyStochastic_2023}, 
which requires evaluating an expected value. The authors also propose an approximate 
reformulation based on Jensen's inequality, but since the expectation can be computed 
exactly in our setting, this approximation is unnecessary. In this formulation, a 
hyperparameter \(\alpha \in (0,1)\) must be selected.

Under the policy \(\pi\) generated by this method, 
\cite[Thm.~5]{cosner_RobustSafetyStochastic_2023} guarantees
\(P_N(x;\pi) \le 1 - \tfrac{h(x)}{\bar{h}}\, \alpha^N\),
where \(h\) is assumed to be bounded above by 
\(\bar{h} = \sup_x h(x)\). 
For \(h \equiv \hQ\), we have \(\bar{h} = 0.5^2\).
We select \(\alpha\) by equating the right-hand side to the desired 
exit probability \(\epsilon\) and solving for \(\alpha\).

The \fushimi{} method refers in particular to \cite[Thm. 11]{fushimi_SafetyCriticalControlDiscretetime_2025}, 
which involves two hyperparameters ($a$ and $\betaF$).
We select these hyperparameters by following the authors'
approach in \cite[\S 5.2]{fushimi_SafetyCriticalControlDiscretetime_2025}.
For a thusly defined policy $\pi$,
the exit probability is bounded by \cite[eq. (48)]{fushimi_SafetyCriticalControlDiscretetime_2025}
\begin{equation}\label{Pfushimi}
	P_N(x;\pi) \leq \hat{P}(x; a, \betaF) = \exp(-a h(x))+\betaF N.
\end{equation}
Here, the parameters $a \geq 1$ and $\betaF \geq 0$ are to be selected so that
\begin{inlinelist*}
  \item\label{item:feasible} the respective safety filter condition (\cite[eq. (47)]{fushimi_SafetyCriticalControlDiscretetime_2025})
    is feasible for all \(x \in S\)
  \item\label{item:tight} \(\hat{P}(x; a, \betaF)\) is (tightly) upper-bounded by \(\epsilon\).
\end{inlinelist*}
In this example, \cref{item:feasible} is satisfied if
\begin{equation}\label{eq:fushimi-feas}
  a \bar{h} \geq
-\log(\exp(-a \bar{h})+\betaF)
-\tfrac{1}{2}\log\det\!\left(I + \Sigma aA\right),
\end{equation}
where \( A = \diag (0, -2, 0 )\)
is the Hessian of \(h(x)\).
\Cref{item:feasible,item:tight} can thus be achieved by solving
\[ 
  \maximize_{a \geq 1, \betaF \geq 0} \hat{P}(x; a, \betaF) \sttshort \hat{P}(x; a, \betaF) \leq \epsilon, \eqref{eq:fushimi-feas} \text{ holds.}
\]
This problem may have different solutions, leading to different empirical 
exit probabilities. For instance, for any $a \geq 1$, we can
solve $\hat{P}(x; a, \betaF) = \epsilon$ analytically for $\betaF$. 
Any such combination that satisfies \eqref{eq:fushimi-feas} is 
an optimal solution.

We will compare the described methods on two main metrics:
\begin{inlinelist*}
	\item the tightness of the exit probability bound
	\item feasibility as a function of the tolerance \(\epsilon\) and noise level \(\sigma\)
\end{inlinelist*}
To evaluate tightness, we use Monte Carlo simulation to estimate the exit probability 
starting from \(x_0 = 0\), with noise level \(\sigma = 0.03\) and different
exit probability tolerances \(\epsilon\). %
The results are reported in \cref{tab:tightness-results}.
Although all methods achieve the desired 
safety levels, the proposed method 
more closely approximates the desired tolerance $\epsilon$, 
despite the naive selection of $\delta^{(i)} = \nicefrac{\delta}{\np}$.
In this example, this choice is expected to overestimate the empirical failure probability by approximately a factor of~2.

\begin{table}[ht]
  \vspace{7pt}
	\centering
	\caption{Empirical exit probabilities over 5000 independent runs. \\(\(N=20\), \(x_0=0\), \(\sigma=0.03\), known noise distribution)}
	\label{tab:tightness-results}
	\renewcommand{\arraystretch}{1.25}
  {
  \setlength{\tabcolsep}{3pt}
	\begin{tabular}{lcccc}
		\toprule
		 & \(\epsilon = 0.2\)
		 & \(\epsilon = 0.1\)
		 & \(\epsilon = 0.01\)
		 & \(\epsilon = 0.001\)                            \\
		\midrule
		\cosner{} \cite{cosner_RobustSafetyStochastic_2023}
		 & 0.0036               & 0.0002 & /      & /      \\
		\fushimi \cite{fushimi_SafetyCriticalControlDiscretetime_2025} (a=20)
		 & 0.0434                & 0.0132 & 0 & / \\
		\fushimi \cite{fushimi_SafetyCriticalControlDiscretetime_2025} (a=35)
		 & 0.0602               & 0.0276 & 0.0014 & 0      \\
		\fushimi \cite{fushimi_SafetyCriticalControlDiscretetime_2025} (a=50)
		 & 0.0544               & 0.0278 & 0.0024 & 0.0002 \\
		Ours       & \textbf{0.0766}               & \textbf{0.0374} & \textbf{0.0046} & 0.0002 \\
		\bottomrule
	\end{tabular}
  }
	\begin{minipage}[t]{\linewidth}
    \vspace{1pt}
		\raggedright
		{\footnotesize
     `/' signifies infeasibility.
		}
	\end{minipage}
\end{table}

Next, we test feasibility of the methods from initial state
\(
x_0 = 0
\).
In \cref{fig:infeasibilities}, the feasibility results for different 
$\sigma$--$\epsilon$ combinations are shown.
The proposed method yields the largest feasible region,
indicating that it is able to handle higher amounts of noise with the same guarantees.

\begin{figure}[htb]
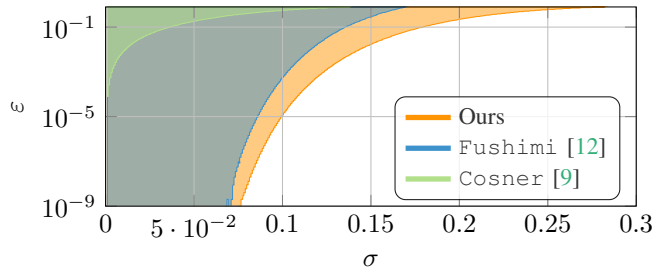

	\centering
	\includetikz{infeas_finegrain}
  \vspace{-7pt}
	\caption{Combinations of the noise standard deviation $\sigma$ and the tolerated 20-step exit probability $\epsilon$
		for which the different methods are feasible.
	}
  \vspace{-5pt}
	\label{fig:infeasibilities}
\end{figure}

\subsubsection{Unknown disturbance distribution} \label{sec:unknown-dist}
We now present the data-driven variant
and compare it to the conformal and scenario approach in
\cite{mestres_ProbabilisticControlBarrier_2025}.
The setting is identical to before, except that the distribution of
$\xi$ is assumed to be unknown and we instead have a data sample
$\mathcal{D} = (\xi_t)_{t \in [\nsamp]}$
with $\xi_t \sim \normal(0, \Sigma)$.
\Cref{tab:compare-dd} compares the tightness and computational complexity of the different methods for \(\gamma = 0.01\) (as in \cite{mestres_ProbabilisticControlBarrier_2025}),
corresponding to \(\tilde{\gamma} = 0.2\), where \(\nsamp\) is chosen such that all methods are feasible.
Since the scenario approach requires solving \acp{QCQP}, we use \Clarabel{} \cite{goulart2024clarabel} for both this method and ours,
while the mixed-integer formulation of the conformal approach is solved using \Gurobi{} \cite{gurobi} through its Python interface.
Note that even though we use the quantile estimator of \cref{prop:data-driven-ij-indep},
we do recompute the quantile in every time step to simulate the computational 
growth in the general case of \cref{prop:data-driven-ij}.

We observe that our method remains significantly less conservative and less computationally 
demanding. The conformal approach becomes prohibitively expensive as $\epsilon$ decreases.
The main reason for this is that the conformal quantile estimator
depends on the decision variable, since the safe set is modeled using a quadratic $h$,
yielding a mixed-integer problem with a number of binary variables increasing
linearly with the sample size \(\nsamp\).
In contrast, in our method,
the quantile estimator is independent of the decision variable (cf. \eqref{eq:constraint}).
For the scenario approach, the number of constraints increases
linearly with the sample size \(\nsamp\), resulting in a more manageable growth of
its computation time. Nevertheless, it is not real-time capable for large sample sizes $\nsamp$.

\begin{table}[htb]
	\centering
	\caption{
		Comparison with \cite{mestres_ProbabilisticControlBarrier_2025} over 500 independent runs.
		(\(\tilde{\gamma}=0.2\),~\(N=20\),~\(x_0=0\),~\(\sigma=0.03\),~unknown distribution)}
	\label{tab:compare-dd}
	\begin{tabular}{llcccc}
		\toprule
		Method
		     & Metric
		     & \(\epsilon = 0.9\) & \(\epsilon = 0.5\) & \(\epsilon = 0.1\)
		\\
		\midrule
		     & \(\nsamp\)         & 254                & 2580               & 108430               \\
		\midrule
		Conform. \cite{mestres_ProbabilisticControlBarrier_2025}
         & $\hat{P}_{N}$      & 0.164               & 0.058            & /\textsuperscript{$\star$}                    \\
   & $T_{\text{avg}}$   & 27.79              & \num{1.05e3}       & \num{1.85e6} \\
		     & $T_{\text{max}}$   & 244.0             & \num{9.52e3}       & \num{9.75e6} \\
		\midrule
		Scen. \cite{mestres_ProbabilisticControlBarrier_2025}
         & $\hat{P}_{N}$      & 0.058               & 0.004            &       0             \\
    & $T_{\text{avg}}$   &1.12              & 4.91       &  390.27\\
		     & $T_{\text{max}}$   & 5.9             & 9.89       & 539.78 \\
		\midrule
		Ours & $\hat{P}_{N}$      & 0.436              & 0.206              & 0.038                \\
      & $T_{\text{avg}}$   & 0.86               & 0.87               & 1.33                 \\
		     & $T_{\text{max}}$   & 2.99            & 7.11              & 3.19               \\
		\bottomrule
	\end{tabular}
	\begin{minipage}[t]{\linewidth}
		\vspace{2pt}
		{\raggedright
      \footnotesize
			$T_{\text{avg}}$, $T_{\text{max}}$: average and maximum execution times [\si{\milli\second}]\\
			$\hat{P}_N$: empirical exit probabilities.\\
			$^\star$Stopped after 10 time steps due to excessive computation time.
		}
	\end{minipage}
\end{table}

Next, we validate our data-driven approach by simulating
under different heavy-tailed distributions, namely
$\xi_t \sim \laplace(0, \Sigma_{\text{H}})$, with
$\xi_t \sim \student_{\nu}(0, \Sigma_{\text{H}})$, with
$\Sigma_{\text{H}} = \diag(0, \sigma, 0)$, and $\nu = 8$.
\Cref{fig:coverage-heavy-tails} shows the
empirical exit probabilities $\hat{P}_N$ 
over 2000 independent runs (using $\tilde{\gamma} = 0.01$, $\epsilon =0.1$).
As expected, we observe that as $\nsamp$ increases, 
$\hat{P}_N$ increases towards $\epsilon$
due to the binomial distribution involved in \cref{prop:data-driven-ij} 
concentrating more closely around its mean.
It is clear that under more heavily-tailed distributions,
the violation probabilities lie slightly closer to the tolerance.
Since no tail assumptions are made by the estimator, this effect is expected
to be more pronounced for even more heavy-tailed distributions.
\begin{figure}[ht!]
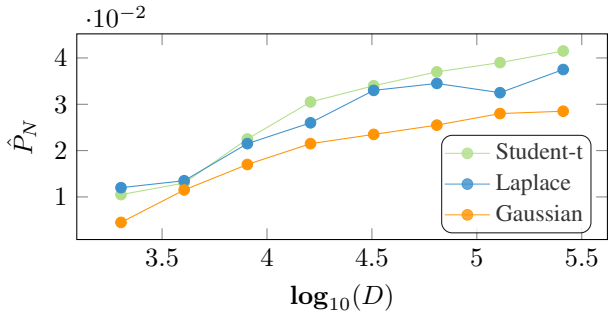

	\centering
	\includetikz{coverage_heavy_tails}
  \caption{Empirical exit probabilities $\hat{P}_N$ in the experiment of \cref{sec:unknown-dist}.
	}
	\label{fig:coverage-heavy-tails}
\end{figure}

\subsection{General path planning}\label{sec:path-planning}

To demonstrate the scalability and modeling flexibility of the
proposed approach, we next consider a larger path planning scenario.
For convenience, we consider the same dynamics as in \cref{sec:corridor}.
However, in this case, we consider a more elaborate environment, consisting
of 13 irregular polyhedral obstacles in $\R^2$, representing a narrow, cluttered
pathway.
The base policy simply tracks a given reference trajectory using a time-invariant
piecewise linear control law, disregarding the obstacles on its path.
\Cref{fig:path-planning-positions}
shows position trajectories of the robot over 30 independent simulations.
We simulate with $\xi_k \sim \normal(0, \Sigma)$, with 
$\Sigma = \diag(\sigma^2, \sigma^2, 0)$, where $\sigma = 0.03$.
The $N$-step exit probability was set to
$\epsilon = 0.1$, with $N = 150$ time steps.
We bound the control actions by $\| u_k \|_{\infty} \leq 5$.
It is apparent from \cref{fig:path-planning-positions} that the
safety filter is able to successfully correct the unsafe trajectories of the
base controller by steering the robot around the obstacles, without
maintaining overly conservative safety margins.
\begin{figure}[htb]
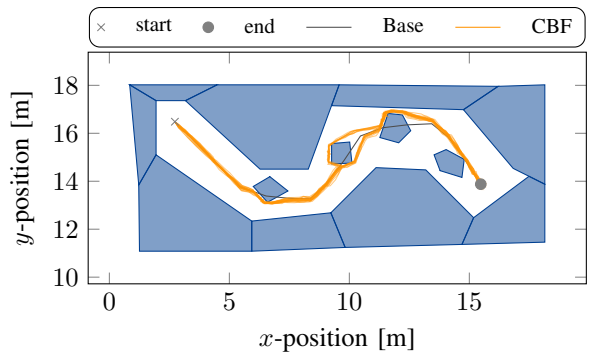

	\centering
	\includetikz{corridor-with-obstacles}
  \vspace{-2pt}
  \caption{State trajectories in the path planning example of \cref{sec:path-planning}.}
	\label{fig:path-planning-positions}
  \vspace{-5pt}
\end{figure}

Using this test case, we will now compare the proposed approximate
scheme (\cref{alg:main-method})
with the alternative approach of directly 
solving the \ac{MIQP} \eqref{eq:mip} using the big-M reformulation.
\newcommand{\theur}{t^{\textrm{QP}}}%
\newcommand{\tmi}{t^{\textrm{MI}}}%
We repeat the path planning experiment with the \ac{MIQP}
formulation and record the time to solve the problem using
\Gurobi{} \cite{gurobi}, resulting
in $T = 150 \times 30 = 4500$ timings $\tmi_{k}$, $k \in [T]$.
Similarly, we denote by $\theur_{k}$, $k \in [T]$,
the measured execution times of the heuristic proposed in
\cref{alg:main-method}, where we solve the QPs using \QPALM{} \cite{hermans2022qpalm}.
We find that $\theur_k$ was
\SI{5.4}{\milli\second} on average with a maximum of \SI{17.3}{\milli\second}
whereas, the $\tmi_k$ was
\SI{7.3}{\milli\second} on average with a maximum of \SI{44.6}{\milli\second}.
We recorded an average speedup $s_k = \nicefrac{\tmi_k}{\theur_k}$ of 1.4, a maximum of 8.5.

\newcommand{\vheur}{v^{\textrm{QP}}}%
\newcommand{\vmi}{v^{\textrm{MI}}}%
On the other hand, \cref{alg:main-method} only computes an approximate solution. 
To measure suboptimality, we simulate the state following the base policy 
\(\bar{u}_k = \bar{\kappa}_k(x_k)\), and for each \(x_k \in S\), we solve the 
safety filter problem using both methods, obtaining costs \(\vheur_k\) and 
\(\vmi_k\), \(k \in [T]\). Let \(\underline{v}_k = \min \{\vheur_k, \vmi_k\}\) 
denote the lowest cost at time \(k\). Discarding time steps where 
\(\underline{v}_k < 10^{-2}\) (i.e., where the safety filter made no substantial 
modification), the relative suboptimality of \cref{alg:main-method} is computed as 
\(\tfrac{\vmi_k - \underline{v}_k}{\underline{v}_k}\).  
We find that in 90\% of the remaining cases, the relative suboptimality is below 
\(10^{-7}\), while the largest value is 16.81, indicating that in rare cases the 
heuristic selects a different set of facets for the polyhedral obstacles, 
resulting in qualitatively different corrections. Note, however, when including all time 
steps, over 99\% record negligible suboptimality (below \(10^{-7}\)).

\section{Conclusion and future work}\label{sec:conclusion}
We proposed a safety filter design for
discrete-time stochastic systems, utilizing
piecewise affine probabilistic \acp{CBF},
which are highly expressive, yet computationally manageable.
We derived tractable sufficient conditions to ensure
a given upper bound on the $N$-step exit probability 
(with both known and unknown noise distributions),
and demonstrated that it is less conservative than comparable methods in the literature.
Furthermore we illustrate the flexibility of the method
for modeling complex safety sets.

\bibliographystyle{hieeetr}
\bibliography{references}

\ifArxiv
	\appendix
	\subsection{Deferred proofs}\label{sec:appendix}
	\begin{proof}[Proof of \cref{lem:bound-single}]
	For $i \in [\np]$, $j \in [\nf_i]$ define the events
	\(
	E_i = \{ \xi \mid \Dh_i (x,u,\xi) \geq 0 \}
	\),
	\(
	E_{ij} = \{ c_{ij}^\top(F(x, u, \xi)) \geq \bt_{ij}(x) \}
	\).
	Then,
	by definitions \eqref{eq:pw-lin-barrier} and \eqref{eq:shorthands},
	\begin{align*}
		\prob_{\xi}(E_i) & = \prob_{\xi}\big\{\textstyle \max_{j \in [\nf_i]}
		c_{ij}^\top(F(x, u, \xi)) - \bt_{ij}(x) \geq 0 \big\}                                                     \\
		                 & = \prob_{\xi} \Big (
                     \textstyle\union_{j \in [\nf_i]} E_{ij} \Big) \geq \max_{j \in [\nf_i]} \prob_{\xi}(E_{ij}).
	\end{align*}
	Therefore,
	$\prob_\xi(\compl{E_i}) =
		1 - \prob_\xi(E_i)
		\leq
		1 -
    \max_{j \in [\nf_i]} \prob_{\xi}(E_{ij}) = \min_{j \in [\nf_i]} \prob_{\xi}(\compl{E_{ij}}),
	$
  proving the claim.
\end{proof}

\begin{proof}[Proof of \cref{lem:union-bound}]
	Define the event $\bar{E}_i = \{ \Dh_i(x, u, \xi) < 0 \}$.
	By the union bound, we have
	\begin{align}
    \hspace{-8pt}\nonumber\prob_{\xi} \{ \Dh(x, u, \xi) < 0 \} & = \prob_{\xi} \{ \min_{i \in[\np]} \Dh_i(x,u,\xi) < 0 \}          \\
                                                  & = \prob_{\xi} \big( {\textstyle \union_{i \in [\np]}}  \bar{E}_i \big)
                                                  \leq \tsum_{i \in [\np]} \prob_{\xi} (\bar{E}_i).
		\label{eq:pf-union-bound}
	\end{align}
  By definitions \eqref{eq:pw-lin-barrier}, \eqref{eq:shorthands}, 
  event $\bar{E}_i$ can be characterized by
  $\bar{E}_i = \{ F(x, u, \xi) \in P_i \}$, which implies that
	$\bar{E}_i \cap \bar{E}_{i'} = \{ F(x,u ,\xi) \in P_i \cap P_{i'} \}$. Thus,
	if the collection $\{P_i\}_{i \in [\np]}$ is pairwise disjoint, $P_{i} \cap P_{i'} = \emptyset$
  for any $i \neq i'$.
	Therefore $\prob_{\xi}(\bar{E}_i \cap \bar{E}_{i'}) = 0$ and \eqref{eq:pf-union-bound}
	holds with equality.
\end{proof}

\begin{proof}[Proof of \cref{lem:quantile-bound}]
	We define $N(q) \dfn | \{ i : Z_i \leq q \} |$. Note that $N(q) \sim \bin(\nsamp ; F_Z(q))$.
	By definition of $N(q)$:
	\begin{equation*}
		\begin{aligned}
			\prob \{ Z_{(k)} \leq \quantile[\delta](Z) \} & = \prob \{ N(\quantile[\delta](Z)) \geq k \}       \\
			                                              & = 1 - \prob\{ N(\quantile[\delta](Z)) \leq k-1 \}  \\
			                                              & = 1 - \Fb(k-1; \nsamp, F_Z(\quantile[\delta](Z))). \\
			                                              & \geq 1 - \Fb(k-1; \nsamp, \delta),
		\end{aligned}
	\end{equation*}
	where $k \mapsto \Fb(k; \nsamp, \delta)$ is the \ac{CDF} of $\bin(\nsamp, \delta)$,
	and the inequality follows from the fact that
	\begin{inlinelist*}
		\item $F_Z(q) \geq \delta$
		\item $p \mapsto \Fb(k; \nsamp, p)$ is nonincreasing for all
		$k \leq \nsamp$.
	\end{inlinelist*}
	Thus, if there exists a $k^\star \geq 1$ such that
	\(
		\Fb(k^\star-1; \nsamp, \delta) \leq \gamma,
	\)
	then $ \prob \{ Z_{(k^\star)} \leq \quantile[\delta](Z) \} \geq 1 - \gamma$.
	Since $k \mapsto \Fb(k; \nsamp, \delta)$ is nondecreasing, such a $k^\star$
  exists if and only if
	\(
			\Fb(0; \nsamp, \delta) = (1 - \delta)^{\nsamp} \leq \gamma
			{\iff}                    \nsamp \geq \tfrac{\log(\gamma)}{\log(1 - \delta)}.
	\)
  Further, if this inequality is strict,
  then $k^\star = \Fb^{-1}(\gamma; \nsamp, \delta) \geq 1$ and  
	$\Fb(k^\star-1; \nsamp, \delta) \leq \gamma$.
\end{proof}

\begin{proof}[Proof of \cref{prop:data-driven-ij}]
  Fix $x \in \R^{\ns}$,
  and let 
	$E_{ij}(x) =
  \{ Z_{(k_i)}^{ij} \leq \fullquant{x, \xi} 
  \}$
	for $i \in [\np]$ $j \in [\nf_i]$.
	By the union bound, we have
	\[
		\begin{aligned}
			\prob_{\xi}^{\nsamp} \big( \itrsect_{i \in [\np]} \itrsect_{j \in [\nf_i]} E_{ij} \big)
			 & = 1 - \prob_{\xi}^{\nsamp} ( \union_{i \in [\np]} \union_{j \in [\nf_i]} \compl{E_{ij}} ) \\
			 & \geq 1 - \tsum_{
       i \in [\np], j \in [\nf_i]
     } \prob_\xi^{\nsamp}(\compl{E_{ij}})
			\geq 1 - \gamma.
		\end{aligned}
	\]
	where the second inequality follows from \cref{lem:quantile-bound}.
\end{proof}

\fi
\end{document}